\documentclass[12pt]{amsart}
\usepackage[margin=1in]{geometry}
\usepackage{amsmath,amsfonts,amsthm,amssymb,bbm}
\usepackage{graphicx,color,dsfont}
\usepackage{enumitem}
\usepackage{fourier}

\newtheorem{theorem}{Theorem}
\newtheorem{lemma}{Lemma}
\newtheorem{proposition}{Proposition}

\newtheorem{definition}{Definition}
	
	\newtheorem{corollary}{Corollary}

\newtheorem{claim}{Claim}

 \theoremstyle{definition}
 
 \theoremstyle{remark}

 \numberwithin{equation}{section}

\newcommand{\vertiii}[1]{{\left\vert\kern-0.25ex\left\vert\kern-0.25ex\left\vert #1
    \right\vert\kern-0.25ex\right\vert\kern-0.25ex\right\vert}}

\newcommand{\C}{{\mathbb C}}

\renewcommand{\[}{\left[}

\newcommand{\f}[2]{\frac{#1}{#2}}

\newcommand{\cl}{{\mathcal L}}

\newcommand{\al}{\alpha}
\newcommand{\be}{\beta}

\newcommand{\ga}{\gamma}

\newcommand{\ka}{\kappa}
\newcommand{\la}{\lambda}

\newcommand{\La}{\Lambda}
\newcommand{\si}{\sigma}

\newcommand{\vp}{\varphi}

\newcommand{\om}{\omega}

\newcommand{\rone}{\mathbb R}

\newcommand{\rthree}{\mathbf R^3}

\newcommand{\dpr}[2]{\langle #1,#2 \rangle}

\newcommand{\eps}{\epsilon}

\newcommand{\ch}{\mathcal H}

\newcommand{\p}{\partial}

\newcommand{\beq}{\begin{equation}}
\newcommand{\eeq}{\end{equation}}
\newcommand{\beqna}{\begin{eqnarray*}}
\newcommand{\eeqna}{\end{eqnarray*}}
\newcommand{\beqn}{\begin{equation*}}
\newcommand{\eeqn}{\end{equation*}}
\newcommand{\bp}{\begin{proof}}
\newcommand{\ep}{\end{proof}}
\newcommand{\bprop}{\begin{proposition}}
\newcommand{\eprop}{\end{proposition}}
\newcommand{\bt}{\begin{theorem}}
\newcommand{\et}{\end{theorem}}
\newcommand{\bex}{\begin{Example}}
\newcommand{\eex}{\end{Example}}
\newcommand{\bc}{\begin{corollary}}
\newcommand{\ec}{\end{corollary}}
\newcommand{\bcl}{\begin{claim}}
\newcommand{\ecl}{\end{claim}}
\newcommand{\bl}{\begin{lemma}}
\newcommand{\el}{\end{lemma}}

\newcommand{\cj}{{\mathcal J}}
\newcommand{\ci}{{\mathcal I}}

\begin{document}

\title
[Periodic waves of the Benney system]{On the stability of the periodic waves  for the  Benney system}

 \author{Sevdzhan Hakkaev}
 \author{Milena Stanislavova}
 \author{Atanas Stefanov}

 \address{Sevdzhan Hakkaev,
 Department of Mathematics, Faculty of Science, 	Trakya University,  Edirne, Turkey and  Institute of Mathematics and  Informatics, Bulgarian Academy of Sciences, 1113 Sofia, Bulgaria}
 \email{s.hakkaev@shu.bg}

 \address{Milena Stanislavova,
 	Department of Mathematics,
 	University of Alabama - Birmingham,
 	University Hall, Room 4005,
 	1402 10th Avenue South
 	Birmingham AL 35294-1241, USA  } \email{mstanisl@uab.edu}
 \address{Atanas Stefanov, 	Department of Mathematics,
 	University of Alabama - Birmingham,
 	University Hall, Room 4005,
 	1402 10th Avenue South
 	Birmingham AL 35294-1241, USA }

 \email{stefanov@uab.edu}

 \thanks{
 	Milena Stanislavova is partially supported by NSF-DMS,    \# 2108285.
 	Atanas Stefanov acknowledges  partial support  from NSF-DMS,   \# 1908626.}



\date{\today}

\begin{abstract}
  We analyze the Benney model  for  interaction of short and long waves in   resonant water wave interactions. Our particular interest is in the periodic traveling waves, which we construct and study in detail. The main results are that, for all natural values of the  parameters, the periodic dnoidal waves  are spectrally stable with respect to perturbations of the same period.  For another natural  set of parameters, we construct the snoidal waves, which exhibit instabilities, in the same setup.

  Our results are the first instability results in this context. On the other hand, the spectral stability established herein  improves significantly  upon the work \cite{AVH}, which established stability of the dnoidal waves, on a subset of parameter space, by  relying  on the Grillakis-Shatah theory.  Our approach, which turns out to give definite answer for the entire domain of parameters,  relies on the instability index theory, as developed by \cite{KKS1, KKS2, LZ, Pel}.   Interestingly, end even though the linearized operators are explicit, our spectral analysis requires subtle and detailed analysis of matrix Schr\"odinger operators in the periodic context, which support some interesting features.
\end{abstract}

\maketitle

  \section{Introduction}
  We consider the following  system of PDE
  \begin{equation}\label{1.1}
  \left\{ \begin{array}{ll}
  iu_{t}+u_{xx}=uv+\beta |u|^2u, \ \ -T\leq x\leq T, t\in\rone_+\\
  v_{t}=(|u|^2)_{x},
  \end{array} \right.
  \end{equation}
  where $\beta$ is a real parameter, $u$ is complex valued function, and $v$ is real-valued function. This system is introduced by Benney, \cite{B1, B2} which models the interaction of short and long waves in resonant water waves interaction in a nonlinear medium.

  The Cauchy problem on the whole line case for the system (\ref{1.1}) was considered in \cite{BOP, Co}. The existence and nonlinear stability of solitary waves was studied in \cite{GC, L}.

  We consider such model on a periodic background, that is, we impose a periodic boundary conditions. The Cauchy problem for  \eqref{1.1}  has been previously
  considered in this context, \cite{AVH}. Let us  pause for the moment and review the said paper, as it serves as a starting point for our investigation.  More precisely, in \cite{AVH}, the authors have established, via the Fourier restriction method,  that the problem is locally well-posed for data $(u_0, v_0)\in H^r[-T,T] \times H^s[-T,T]$, whenever
  $\max(0, r-1)\leq s\leq \min(r, 2r-1)$. In particular, Hadamard
  well-posedness holds in the spaces $H^{\f{1}{2}}([-T,T])\times L^2[-T,T]$ and also in the smaller space  $H^{1}([-T,T])\times L^2[-T,T]$. Interestingly, ill-posedness results (in the sense of non-uniformly continuous dependence on initial conditions)  were also obtained in $H^r\times H^s$, whenever $r<0$.

    Here we consider the spectral stability of periodic traveling waves of dnoidal and snoidal type. We are interested in the stability of periodic traveling wave
solutions of (\ref{1.1}) with respect to perturbations that are periodic
of the same period as the corresponding wave solutions.

 We provide the relevant definitions of the various notion of stability below, but we would like to discuss the advances made in the last forty years in the area of stability of periodic traveling waves. Benjamin, in the seminal work, \cite{Ben}, first considered the stability of the cnoidal solution as a periodic traveling wave of KdV. His results were later clarified and streamlined in \cite{ABS}, where the authors have made use of the Grillakis-Shatah-Strauss formalism. It is worth mentioning the work \cite{Ang1}, where the author has addressed, in a similar manner, periodic waves for mKdV and NLS. In the important works \cite{DK1, DK2},  the authors have considered the stability of more general families of solutions arising in the generalized KdV models.
 
 More recently, in the works \cite{NA1, NA2}, Angulo and Natali have developed a novel approach for studying periodic traveling waves for a general class of dispersive models, which extracts the necessary spectral information, based on the so-called positivity theory for the multipliers.  For other models such as Klein-Gordon-Schrodinger system, Schr\"odinger-Boussinesq system and Schrodinger system stability of periodic waves is obtained in \cite{NP1, NP2, NP3, PA, FP}. In the context of standing waves,  interesting contributions were made by Gallay and Haragus, \cite{GH1} and \cite{GH2}. While the results in \cite{GH1} concern periodic waves in the context of NLS on the line, the results in  \cite{GH2} are more relevant to our discussion herein. Namely,  rigorous stability analysis was developed to deal with quasi-periodic waves in the cubic NLS context, both in the focussing and defocussing scenarios. All of these works, rely, in one degree or another on the Grillakis-Shatah-Strauss approach, which establishes orbital stability based on conservation laws. This almost universally requires a $C^1$ dependence on the wave speed parameters, which is not always easy to establish, so an ad hoc assumption in that regard is  usually made.

 As it turns out, one may study an almost equivalent stability property, namely the spectral/linear stability, see Definition \ref{defi:10} below\footnote{In fact, under some generic  conditions  on the waves, one may convert such spectral stability statements into orbital stability results, see Theorem 5.2.11, \cite{Kap}}. This is a fast developing theory, which has seen some spectacular advances in recent years, \cite{KKS1, KKS2, LZ, Pel}. This approach, has several   advantages over the classical GSS approach. For example, one can study the spectral stability as a purely linear problem, without paying particular attention to the actual conservation laws, see \eqref{124} below. A second major advantage is that, when it comes to systems of coupled PDE's, it is just technically hard to deal with the conservation laws directly, as the linearized operators become non-diagonal matrix operators, which are harder to analyze.

 The stability of waves, especially in the context of systems of coupled PDE, especially in the spatially periodic context,  is a challenging topic and an active area of research.   We should point out that great progress was made in the last fifteen years regarding dispersive equations for scalar quantities  - in that regard, we would like to mention the works \cite{BJ1}, \cite{BJK1} for KdV type models, while \cite{BJK2} established an index counting formula for abstract second order in time models.  Concerning systems of dispersive PDE,
 there are just a few results available in the literature about periodic waves. In fact, we are aware of just a few  rigorous works on the subject - \cite{HSS} deals with stability of periodic waves in systems by the index counting method, while \cite{HSS1} and \cite{FP} apply the more standard GSS formalism to the corresponding problem at hand. One explanation  for the relative scarcity of rigorous analytical results in this context  are the difficulties associated with the  spectral analysis of the linearized operators  in cases of systems.

Regarding the Benney system, which is the system of interest in this article,   it was already considered in \cite{AVH}. More specifically, the authors were able to construct a family  of smooth periodic traveling waves of dnoidal type and show their  orbital stability. This was done under certain conditions on $\be$ and by relying on the Grillakis-Shatah-Strauss approach. More specifically, they   rely  on  the following conservation laws for the Benney system,
 $$\begin{array}{ll}
 M(u)=\int_{-T}^{T} {|u(t,x)|^2}dx\\
 \\
 E(u,v)=\int_{-T}^{T} \left[v(t,x)|u(t,x)|^2+|u_x(t,x)|^2+\frac{\beta}{2}|u(t,x)|^4 \right]dx\\
 \\
 P(u,v)=\int_{-T}^{T} {\left[ |v(t,x)|^2+2\Im (u(t,x)\overline{u}_x(t,x))\right] }dx.
 \end{array}
 $$
 In order to explain
our spectral  stability results in detail, we need to linearize the system (\ref{1.1})  about
the periodic traveling wave solutions. Then we need to obtain the required
spectral information about the operator of linearization and investigate the
index of stability $k_{Ham}$, as introduced  in \cite{LZ}.

The paper is organized as follows. First, we construct the periodic traveling waves of dnoidal and snoidal type and set-up the linearized problem for system (\ref{1.1}). In Section 2, we overview the index
stability theory and investigate spectral properties of the operator of the linearization. In Section 3, using the index counting theory we analyze the stability of periodic traveling waves.

  \subsection{Periodic traveling waves}
   In this section, we construct periodic waves of the form
   $$
   u(t,x)=
   e^{i\omega t}e^{i\frac{c}{2}(x-ct)}\varphi(x-ct), \ \
   v(t,x)=\psi(x-ct),
   $$
   for the Benney
   system (\ref{1.1}). Plugging in (\ref{1.1}), we get the following system
   \begin{equation}\label{2.1}
   \left\{ \begin{array}{ll}
   \varphi''-\left( \om-\frac{c^2}{4}\right)\varphi=\varphi \psi+\beta\varphi^3 \\
   -c\psi'=2\varphi \varphi'
   \end{array} \right.
   \end{equation}
   The case $c=0$ leads to semi-trivial constant solutions $\vp$, so we do not consider it herein. Henceforth, we assume $c\neq 0$.
 Integrating second equation in (\ref{2.1}),
 we get $\psi=-\frac{1}{c}\varphi^2+\ga$, where $\ga$ is a constant of integration. Substituting $\psi$ in the first equation of (\ref{2.1}), we get the following equation for $\varphi$
 \begin{equation}\label{2.2}
 \varphi''-\sigma\varphi=\left( \beta -\frac{1}{c}\right) \varphi^3,
 \end{equation}
  where  we have introduced the important parameter $\sigma =\om -\frac{c^2}{4}+\ga$.  Integrating, we get \begin{equation}\label{2.3}
  \varphi'^2=\frac{1}{2}\left(\beta -\frac{1}{c}\right)\varphi^4+\sigma \varphi^2+a=:U(\varphi),
  \end{equation}
  where $a$ is a constant of integration. It is well known, that $\varphi$ is a periodic function provided that the energy level set $H(x; y) = a$ of
  the Hamiltonian system $dH = 0$,
  $$
  H(x; y) = y^2 -\sigma x^2 +\frac{1}{2}\left( \frac{1}{c}-\beta\right)x^4
  $$
  contains an oval (a simple closed real curve free of critical points).
Depending on the properties of the bi-quadratic polynomial $U(\vp)$, we distinguish two cases, which give rise to different explicit solutions, both in term of the Jacobi elliptic functions.
\subsubsection{Dnoidal solutions}
Consider the case  $\frac{1}{c}-\beta>0$, $\sigma>0$, and $a<0$. Denote by $\varphi_0>\varphi_1>0$,  the positive roots of
$-\varphi^4+\frac{2c\sigma}{1-c\beta} \varphi^2+\frac{2ca}{1-c\beta}$.
Then, the profile equation \eqref{2.3} takes the form
$$
\varphi'^2=\frac{1-c\beta}{2c}(\varphi_0^2-\varphi^2)(\varphi^2-\varphi_1^2)
$$
Then $\varphi_1<\varphi<\varphi_0$ and up to translation the solution $\varphi$ is given by
\begin{equation}\label{2.4}
\varphi (x)=\varphi_0 dn (\alpha x, \kappa),
\end{equation}
where
\begin{equation}\label{2.5}
\varphi_0^2+\varphi_1^2=\frac{2c\sigma}{1-c\beta}, \; \; \kappa^2=\frac{\varphi_0^2-\varphi_1^2}{\varphi_0^2},
\; \; \alpha^2=\frac{1-c\beta}{2c}\varphi_0^2=\frac{\sigma}{2-\kappa^2}.
\end{equation}
Since the period of $dn$ is $2K(\kappa)$, then the fundamental period of $\varphi$ is $2T=\frac{2K(\kappa)}{\alpha}$.

The next case of consideration are the snoidal solutions.
 \subsubsection{Snoidal solutions}
 Let  $\frac{1}{c}-\beta<0$, $\sigma<0$ and $a<0$. Then
 $$
 \varphi'^2=\frac{c\beta-1}{2c}(\varphi_0^2-\varphi^2)(\frac{2c\sigma}{1-c\beta}-\varphi_0^2-\varphi^2).$$
 Up to translations the solution is given by
 \begin{equation}\label{2.8}
 \varphi(x)=\varphi_0sn(\alpha x, \kappa),
 \end{equation}
 where
 \begin{equation}\label{2.9}
 \kappa^2=\frac{(1-c\beta)\varphi_0^2}{2c\sigma-(1-c\beta)\varphi_0^2}, \; \; \alpha^2=-\frac{2c\sigma-(1-c\beta)\varphi_0^2}{2c}=-\frac{\sigma}{1+\kappa^2}.
 \end{equation}
 Since the period of $sn$ is $4K(\kappa)$, then the fundamental period of $\varphi$ is $2T=\frac{4K(\kappa)}{\alpha}$.

We formulate our findings in the following proposition.
\begin{proposition}
\label{prop:ex}
Let $(c, \be, \si)$ are three real parameters and $\ka\in (0,1)$.  Then, we can identify the following families of solutions of \eqref{2.3}.

If $c\neq 0$ and $\be<\f{1}{c}, \si>0$, then $\vp$ is a  family of dnoidal solutions given by \eqref{2.4}. Its parameters are given by
\begin{equation}
\label{301}
\vp_0^2=\f{2 \si}{(2-\ka^2) (\f{1}{c}- \be)},  \ \
\alpha^2=\frac{\sigma}{2-\kappa^2}. \ \
\end{equation}
whereas its fundamental period is $2T=\frac{2K(\kappa)}{\alpha}= \frac{2K(\kappa) \sqrt{2-\ka^2}}{\sqrt{\si}}$. Note that this is a three free parameter family, depending and uniquely determined by $(\f{1}{c}-\be , \si, \ka)\in \rone_+\times \rone_+\times (0,1)$.

If $c\neq 0$ and $\be>\f{1}{c}, \si<0$, we obtain the snoidal
family described in \eqref{2.8}, where
\begin{equation}
\label{302}
\vp_0^2=   \frac{2 \sigma \kappa^2}{(\f{1}{c}- \beta)(1+\kappa^2)}, \ \
\alpha^2= -\f{\si}{1+\ka^2},\ \
\end{equation}
and fundamental period given by $2T=4K(\ka) \f{ \sqrt{1+\ka^2}}{\sqrt{-\si}}$. This is also uniquely determined by three independent parameters as follows
$(\f{1}{c}-\be , \si, \ka)\in \rone_-\times \rone_-\times (0,1)$.
\end{proposition}

 Now that we have identified the relevant nonlinear waves for the Benney model \eqref{1.1}, we focus our attention to the corresponding  linearized problem.
 \subsection{Linearized equations}
 We take the perturbation in the form
 \begin{equation}\label{l1}
 u(t,x)=e^{i\omega t}e^{i\frac{c}{2}(x-ct)}(\varphi(x-ct)+U(t,x-ct)), \; \;
 v(t,x)=\psi(x-ct)+V(t,x-ct)
 \end{equation}
 where $U(t,x)$ is complex valued function, $V(t,x)$ is real valued function. Plugging in the system (\ref{1.1}), using (\ref{2.1}), and
 ignoring all quadratic and higher order terms yields a linear equation
 for $(U,V)$.
 Furthermore, we split the  real and imaginary parts of complex valued function $U$  as $U=P+iQ$, which allows us to rewrite the linearized problem  as the following system

 \begin{equation}\label{l3}
 \left\{ \begin{array}{ll}
 -Q_t=-P_{xx}+\left( w-\frac{c^2}{4}\right) P+3\beta\varphi^2P+\varphi V+\psi P\\
 P_t=-Q_{xx}+\left( w-\frac{c^2}{4}\right) Q+\psi Q+\beta \varphi^2 Q\\
 V_t-cV_x=2\partial_x(\varphi P).
 \end{array} \right.
 \end{equation}
  Let us denote
  $$
  \mathcal{J}:=\left( \begin{matrix} 0&0&1 \\ 0& 2\partial_x &0 \\-1&0&0 \end{matrix} \right) , \; \; \mathcal{H}:=\left( \begin{matrix}
  L_1 &\varphi &0 \\ \varphi & \frac{c}{2} &0 \\ 0&0&L_2 \end{matrix} \right) ,
  $$
  where\footnote{Note that the operator $L_2$ is the standard operator $L_-$, if we were to consider the waves $\vp$ as solutions to the cubic NLS, see \eqref{2.2}. }
  $$
  \begin{array}{ll}
  L_1=-\partial_x^2+\sigma +\left(3\beta-\frac{1}{c}\right)\varphi^2 \\
  L_2=-\partial_x^2+\sigma +\left(\beta-\frac{1}{c}\right)\varphi^2.
  \end{array}
  $$
  Then the system (\ref{l3}) can be written of the form
  \begin{equation}\label{l4}
  \vec{Z}_t=\mathcal{J}\mathcal{H}\vec{Z}, \ \ \vec{Z}=\left(\begin{array}{c}
  P \\ V \\ Q
  \end{array}\right).
  \end{equation}
 The standard mapping  into a time independent problem $\vec{Z}\to e^{\la t} \vec{z}$ transforms the linear differential equation \eqref{l4} into the eigenvalue problem
  \begin{equation}
  \label{124}
  \mathcal{J}\mathcal{H}\vec{z}=\la \vec{z}.
  \end{equation}
 By general properties of Hamiltonian systems, and the operators $\cj, \ch$ in particular, if $\la$ is an eigenvalue of \eqref{124}, then so are, $\bar{\la}, -\la, -\bar{\la}$.
 We give now the following standard definition of spectral stability.
 \begin{definition}
 	\label{defi:10}
 	We say that the wave $\vp$ is spectrally unstable, if the eigenvalue problem \eqref{124} has a non-trivial solution $(\vec{u}, \la)$, so that $ \vec{z}\neq 0, \vec{z}\in H^2[-T,T]\times H^1[-T,T]\times H^2[-T,T]$ and $\la: \Re \la>0$.
 	
 	In the opposite case, that is \eqref{124} has no non-trivial solutions, with $\Re \la>0$,  we say that the wave is spectrally stable.
 \end{definition}
 {\bf Remark:} The definition of linear stability is closely related to the one given in Definition \ref{defi:10} for spectral stability. More precisely,  $\vp$ is a linearly stable wave, if the flow of the differential equation (or equivalently the semigroup generated by $\cj \ch$) has  Lyapunov exponent less or equal to  zero. Equivalently,
 \begin{equation}
 \label{LZ:10}
  \limsup_{t\to \infty} \f{\ln \|\vec{U}(t)\|}{t} \leq 0,
 \end{equation}
  for each initial data $\vec{U}(0)\in H^2[-T,T]\times H^1[-T,T]\times H^2[-T,T]$. It is a standard fact that these two notions coincide in the case of periodic domains, due to the fact that the spectrum of $\cj\ch$ consists of eigenvalues only.  A general justification of \eqref{LZ:10},  which applies to our case,  is provided in Theorem 2.2, \cite{LZ}.

  We are now ready to present our main results, which concern the spectral stability of the traveling periodic waves - of dnoidal and snoidal type.
  \subsection{Main results}
  The following is our main result, which concerns the stability of the dnoidal waves identified in Proposition \ref{prop:ex}.
  \begin{theorem}(Stability of the dnoidal waves)
  	\label{theo:10}
  	
  	Let $\om \in\rone$ and $c\neq 0, \be<\f{1}{c}, \si>0$. Then, the Benney sytsem \eqref{1.1} has a family of dnoidal solutions in the form
  	$$
  	(e^{i \om t} e^{i \f{c}{2}(x-c t)}
  	\vp(x-ct), \psi(x-ct)=(e^{i \om t} e^{i \f{c}{2}(x-c t)}
  		\vp(x-ct), -\f{1}{c}\vp^2(x- ct) + \si+\f{c^2}{4}-\om)
  		$$
where the dnoidal solutions $\vp$ are identified by \eqref{2.4}, whose parameters are given by \eqref{301}. These solutions are spatially periodic, provided
\begin{equation}
\label{304}
c \f{K(\ka)\sqrt{2-\ka^2}}{\sqrt{\si}}\in 2\pi {\mathbb Z}.
\end{equation}
Under these assumptions, the periodic waves are spectrally stable, in the sense of Definition \ref{defi:10}, for all values of the parameters, $\om\in\rone , \si>0, \be<\f{1}{c}, \ka\in (0,1)$, subject to \eqref{304}.
 \end{theorem}

 {\bf Remark:} In \cite{AVH}, the authors proved that dnoidal solutions are orbitally stable for $\beta\leq 0$ and for $\beta>0$ and $8\beta\sigma-3c(1-\beta c)^2\leq 0$. This is achieved by evaluating the number of negative eigenvalues of the operator of linearization around the periodic waves and number of positive eigenvalues of the Hessian of $d(\omega,c)=E(u,v)-\frac{c}{4}P(u,v)-\frac{\omega}{2}M(u,v)$. We extend this result herein to the whole domain of the parameters.

Our next result concerns the instability of the snoidal waves, also identified in Proposition \ref{prop:ex}.
\begin{theorem}(Instability of the snoidal solutions)
	\label{theo:20}
	
	Let $\om\in\rone$ and $c\neq 0, \be>\f{1}{c}, \si<0$. Then, the Benney system has a family of snoidal solutions
	$$
	(e^{i \om t} e^{i \f{c}{2}(x-c t)}
	\vp(x-ct), -\f{1}{c}\vp^2(x- ct) + \si+\f{c^2}{4}-\om)
	$$
	where  $\vp$ is described in \eqref{2.8}, together with \eqref{302}. These waves are periodic exactly when
	\begin{equation}
	\label{306}
	c K(\ka) \f{\sqrt{1+\ka^2}}{\sqrt{-\si}}\in \pi {\mathbb Z}.
	\end{equation}
The snoidal periodic waves are spectrally unstable (with at least one real and positive eigenvalue) for all values of the parameters $\om\in \rone, \si<0, \be>\f{1}{c}, \ka\in (0,1)$, subject to \eqref{306}.
\end{theorem}
The plan for the paper, as well as some major points are explained below.  In Section \ref{sec:2}, we introduce the basics of the instability index theory. We also outline well-known results about the scalar  linearized Schr\"odinger operators $L_1, L_2$ identified earlier, as well as a related operator $L$, which plays significant role in our spectral analysis. This allows us to compute the Morse index of the operator $\ch$ as well as the kernel and the generalized kernel of $\cj \ch$, see Proposition \ref{prop1}. In Section \ref{sec:3}, we deploy the instability index theory to reduce matters to the Morse index of a scalar two-by-two matrix $D$. For the dnoidal case, the computations here are involved,  since only one of the  entries of $D$ is (barely) explicitly computable, and it involves the construction of the Green's function for the  Schr\"odinger operator $L^{-1}$. This is however enough to conclude stability. In the snoidal case, one argues by computing selected (easier) quantities in the limit $0<\be-\f{1}{c}<<1$, which allows one to concludes that real  instability exists close to this limit. Then, a continuation argument, coupled with an earlier rigidity argument about\footnote{establishing that the generalized kernel of $\cj\ch$ remains five dimensional and importantly, does not change across the parameter domain} $Ker(\cj\ch)$ confirms that the real instability persists across the whole domain of parameters.
\section{Preliminaries}
\label{sec:2}
 We first review the basics of the instability index theory, as developed in  \cite{KKS1, KKS2, LZ, Pel}.
 \subsection{Instability index count}
 \label{sec:2.1}
 We follow the notations and presentation in \cite{KKS1, KKS2}, but the same results appears in \cite{Pel}, while  the most general version can be found in \cite{LZ}.  Consider the Hamiltonian  eigenvalue problem
 \begin{equation}
 \label{ih}
  \mathcal{I}\mathcal{\cl}u=\lambda u,
 \end{equation}
where  $\ci^*=-\ci, \cl^*=\cl$ and $\ci, \ch: \overline{\ci f}=\ci \bar{f},\ \  \overline{\ch f}=\ch \bar{f}$, i.e. $\ci, \ch$ map real-valued elements  into real-valued elements.

 Introduce the Morse index of a self-adjoint, bounded from below operator $S$, by setting $n(S)=\# \{\la\in \si(S): \la<0\}$, counted with multiplicities. Let $k_r:=\#\{\la\in  \si_{pt.}(\ci \cl): \la>0\}$ represents the number of positive real
 eigenvalues of $\ci\cl$, counted with multiplicities, $k_c:=\#\{\la\in \si_{pt.}(\ci \cl): \Re \la>0, \Im\la>0\}$ - the number of quadruplets of complex eigenvalues of $\ci \cl$ with non-zero
 real and imaginary parts, whereas
 $$
 k_i^-=\#\{i \la, \la>0: \ci\cl f=i \la f, \dpr{\cl f}{f}<0\}
 $$
 is the number of pairs of purely imaginary eigenvalues of
 negative Krein signature.
 Consider the generalized kernel of $\mathcal{J}\mathcal{H}$,
 $$
 gker (\ci \mathcal{\cl})=span\cup_{l=1}^\infty\ker (\mathcal{\ci\mathcal{L}})^l.
 $$
 Under general conditions, described in \cite{KKS1}, one has that $gker (\ci \mathcal{\cl})$ is finite dimensional, so one can take a basis\footnote{In the applications, one  needs to have an explicit form of such a basis anyway, before any determination of the stability can be made. In a way, we shall need to check the finite dimensionality of
 	$gker (\ci \mathcal{\cl})$}, say $\eta_1, \ldots, \eta_N$. Then, we introduce a
 symmetric matrix $D$ by
 $$
 D:=\{ \{D_{ij}\}_{i,j=1}^N \; : \; D_{ij}=\langle \mathcal{\cl} \eta_i, \eta_j\rangle \}.
 $$
 We are now ready to state the main result of this section, namely  the  following formula for the Hamiltonian index,
 \begin{equation}
 \label{ih2}
 k_{Ham}:=k_r+2k_c+2k_{i}^-=n(\mathcal{\cl})-n(D).
 \end{equation}
 Clearly, spectral stability for \eqref{ih} follows from $k_{Ham}=0$, but such a condition is not necessary for spectral stability. For example, one might encounter a situation where $k_{Ham}=2$, but with $k_i^-=1$, which is an example of spectrally stable configuration with a non-zero $K_{Ham.}$. On the other hand, it is clear that if $k_{Ham}$ is an odd integer, then $k_r\geq 1$, guaranteeing instability.

 \subsection{Spectral information about $\cj\ch$}

 Due to the results in Section \ref{sec:2.1}, it becomes clear that we need a determination of a basis of $gker(\cj\ch)$. It turns out that it is helpful to introduce another Schr\"odinger operator, namely
 $$
 L=-\partial_x^2+\sigma +3\left(\beta-\frac{1}{c}\right)\varphi^2.
 $$
 For context, this is the well-known operator $L_+$, if we were to consider the waves as solutions to the standard cubic NLS, see \eqref{2.2}.
 \subsubsection{The spectra  of $L, L_2$}
 For self-adjoint operator $H$ acting on
$L_{per}^2[0; T]$ with domain $D(H) = H^2([0; T])$, we have  that its
spectrum is purely discrete,
$$ \lambda_0 < \mu_0\leq \mu_1 <\lambda_1 \leq \lambda_2 < \mu_2\leq \mu_3<\lambda_3\leq \lambda_4 < ...$$
Eigenvalues $\lambda_i$, $i=0,1,2...$ corresponds to the periodic eigenvalues, while $\mu_i$, $i=0,1,2...$ corresponds to the semi-periodic eigenvalues. Then, we have that $Hf=\lambda f$ has a solution of period $T$ if and only if $\lambda=\lambda_i$, $i=0,1,2,...$ and a solution of period $2T$ if and only if $\lambda=\lambda_i$, $\lambda=\mu_i$, $i=0,1,2,...$.

  We start with the observation that  $L\varphi'=0$, which is obtained by differentiating equation (\ref{2.2}) respect to $x$.  Also, $L_2\vp=0$, which is just a restatement of \eqref{2.2}.  It is actually helpful, for the rest of the argument, to list the lowest few eigenvalues for both operators $L, L_2$, where $\vp$ is either the dnoidal solution \eqref{2.4} or the snoidal solution \eqref{2.8}. In fact, matters reduce to the explicit Hill operators
 \begin{eqnarray*}
 \Lambda_1 &=& -\partial_{y}^{2}+6k^{2}sn^{2}(y, k) \\
 \Lambda_2 &=& -\partial_y^2+2k^{2}sn^{2}(y, k)
 \end{eqnarray*}
 It is well-known that the first four   eigenvalues of $\La_1$
 with periodic boundary conditions on $[0, 4K(k)]$ are
 simple.
 These eigenvalues and corresponding eigenfunctions are given by
 $$
\left\{\begin{array}{ll}
\nu_{0}=2+2\kappa^2-2\sqrt{1-\kappa^2+\kappa^4},
& \phi_{0}(y)=1-(1+\kappa^2-\sqrt{1-\kappa^{2}
	+\kappa^{4}})sn^{2}(y, \kappa),\\[1mm]
\nu_{1}=1+\kappa^{2}, & \phi_{1}(y)=cn(y, \kappa)dn(y, \kappa)
=sn'(y, \kappa),\\[1mm]
\nu_{2}=1+4\kappa^{2}, & \phi_{2}(y)=sn(y, \kappa)dn(y, \kappa)
=-cn'(y, \kappa),\\[1mm]
\nu_{3}=4+\kappa^{2}, & \phi_{3}(y)=sn(y, \kappa)cn(y, \kappa)
=-\kappa^{-2}dn'(y, \kappa).\\[1mm]
\end{array}\right.
$$
 Regarding $\La_2$,  the
 first three eigenvalues and the corresponding eigenfunctions with
 periodic boundary conditions on $[0, 4K(k)]$ are simple and
 $$\begin{array}{ll}
 \epsilon_0=k^2, & \theta_0(y)=dn(y, k),\\[1mm]
 \epsilon_1=1, & \theta_1(y)=cn(y, k),\\[1mm]
 \epsilon_2=1+k^2, & \theta_2(y)=sn(y, k).
 \end{array}
 $$
  In the dnoidal case, using that $\kappa^2sn^2x+dn^2x=1$ and (\ref{2.4}),  (\ref{2.5}), we get
  \begin{equation}
  \label{l:10}
  L=\alpha^2[\Lambda_1-(4+\kappa^2)].
  \end{equation}
  Note that in this case $\nu_0$ and $\nu_3$ corresponds to the periodic eigenvalues, while $\nu_1$ and $\nu_2$ corresponds to the semi-periodic eigenvalues. It follows that the first two eigenvalues of the operator
  $L$, equipped with periodic boundary condition on $[-T,T]$
  are simple, zero is the second eigenvalue, and $n(L)=1$.
  In the snoidal case, using (\ref{2.8}) and (\ref{2.9}), we have
  \begin{equation}
  \label{l:10a}
  L=\alpha^2[\Lambda_1-(1+\kappa^2)].
  \end{equation}

  It follows again that  zero is the second eigenvalue, and $n(L)=1$.

  Regarding the operator $L_2$, in the dnoidal case, using again (\ref{2.4}),  (\ref{2.5}), we have that
  $$
  L_2=\al^2[\La_2-k^2],
  $$
  whence using the spectral information available for $\La_2$, we conclude $L_2\geq 0$, $n(L_2)=0$.

  In the snoidal case, we have
  $$
  L_2=\al^2[\La_2-(1+k^2)],
  $$
  whence the spectral description of $\La_2$ allows us to conclude that $n(L_2)=2$, with a simple eigenvalue at zero. We collect our results about $L, L_2$ in the following proposition.
  \begin{proposition}
  	\label{prop:spec}
  	Let $\vp$ be either the dnoidal wave \eqref{2.4} or the snoidal wave \eqref{2.8}. Then, \begin{itemize}
  		\item In both the dnoidal and snoidal cases, the Hill operator $L$, equipped with periodic boundary conditions on $[-T,T]$, has Morse index $n(L)=1$ and   $Ker[L]=span[\vp']$.
  		\item  In the dnoidal case,  the operator $L_2$ has Morse index $n(L_2)=0$, $Ker[L_2]=span[\vp]$.
  		\item In the snoidal case,  the operator $L_2$ has Morse index $n(L_2)=2$,
  		$Ker[L_2]=span[\vp]$.
  	\end{itemize}
  \end{proposition}
 We are now ready to describe the kernel and the generalized kernel of $\cj\ch$.
 \subsubsection{Generalized Kernel of $\cj \ch$}
  \begin{proposition}
  	\label{prop1}
  	 Let $\vp$ be either the dnoidal wave \eqref{2.4} or the snoidal wave \eqref{2.8}.
  	Then, the kernel of $\mathcal{H}$ is two dimensional, namely
  	\begin{equation}
  	\label{kerh}
  		Ker[\ch]=span[\left(\begin{array}{c} \vp' \\ -\f{2}{c} \vp\vp' \\ 0\end{array} \right),\left(\begin{array}{c} 0 \\ 0\\ \vp \end{array} \right) ].
  	\end{equation}

\noindent  	In addition, under the assumption
  	\begin{equation}
  	\label{ass2}
  	\dpr{L^{-1}\vp}{\vp}\neq 0,
  	\end{equation}
  	we can identify all the generalized eigenvectors as follows
  	\begin{equation}
  	\label{402}
  	gKer(\cj\ch)\ominus Ker(\ch)=span\left[\begin{pmatrix}
  	\f{1}{2c(c\be-1)} \vp \\
  	- \f{\be}{c(c\be-1)} \vp^2 \\
  	\\ L_2^{-1} \vp'.
  	\end{pmatrix}, \begin{pmatrix}
  	-L^{-1}\vp\\
  	\f{2}{c} \vp L^{-1}\vp\\
  	0
  	\end{pmatrix}, \begin{pmatrix}
  	0 \\ 1 \\ 0
  	\end{pmatrix} \right].
  	\end{equation}
  \end{proposition}
 \begin{proof}
 	We start with $Ker[\ch]$.  We have that $\begin{pmatrix}
 	f \\ g \\ h
 	\end{pmatrix} \in \ker \mathcal{H}$ if
 	\begin{equation}
 	\label{s1}
 	\left| \begin{array}{ll}
 	L_1f+\varphi g=0\\
 	\varphi f+\frac{c}{2}g=0\\
 	L_2 h=0
 	\end{array} \right.
 	\end{equation}
 	From the second equation of (\ref{s1}), we have $g=-\frac{2}{c}\varphi f$ and plugging in the first equation, we get
 	$$
 	0=L_1f+\varphi g=-\partial_x^2f+\sigma f+\left(3\beta -\frac{1}{c}\right)\varphi^2-\frac{2}{c}\varphi^2f=Lf
 	$$
 	From Proposition \ref{prop:spec},  we get that all solutions are multiples of $f=\varphi'$ and $g=-\frac{2}{c}\varphi \varphi'$. From Proposition \ref{prop:spec}, we know that $Ker(L_2)=span[\vp]$ and so, from
 	third equation of (\ref{s1}), we have that another vector in $Ker(\ch)$ is  $h=\varphi$. This identifies $Ker(\ch)$ for us as the one presented in \eqref{kerh}.
 	
 	We now turn to a representation for $Ker(\cj\ch)$. Consider $Ker(\cj\ch)\ominus Ker(\ch)$. We set the equations for
 	$\begin{pmatrix}
 	f \\ g \\ h
 	\end{pmatrix} \in Ker(\cj\ch)\ominus Ker(\ch)$. We need to solve $\ch \begin{pmatrix}
 	f \\ g \\ h
 	\end{pmatrix} \in Ker(\cj)=span \begin{pmatrix}
 	0 \\ 1 \\ 0
 	\end{pmatrix} $. This is equivalent to $h=0$ and
 	\begin{equation}
 	\label{a:10}
 	\left| \begin{array}{ll}
 	L_1f+\varphi g=0\\
 	\varphi f+\frac{c}{2}g=1
 	\end{array} \right.
 	\end{equation}
 	Solving it, implies in a similar manner
 	$$
 	f=-\f{2}{c} L^{-1} \vp,\ \  g=\f{2}{c} \left(1+\f{2}{c} \vp L^{-1} \vp\right).
 	$$
 	This yields an additional,   third vector in the representation of $Ker(\cj\ch)$. More specifically, we obtain
 	  	\begin{equation}
 	  	\label{405}
 	  	Ker(\cj \ch)=span\left\{\left(\begin{array}{c}
 	  	\varphi' \\
 	  	-\frac{2}{c}\varphi \varphi'\\
 	  	0
 	  	\end{array}\right), \ \ \begin{pmatrix}
 	  	0 \\ 0 \\ \varphi
 	  	\end{pmatrix}, \begin{pmatrix}
 	  	-L^{-1}\vp \\ 1+\f{2}{c} \vp L^{-1} \vp \\ 0
 	  	\end{pmatrix} \right\}.
 	  	\end{equation}
 	We now work on identifying the adjoint/generalized eigenvectors. We start with the next level adjoints e-vectors, namely  $Ker((\cj\ch)^2)$. First, we consider the equation
 	$$
 	\cj\ch \begin{pmatrix}
 	f \\ g\\ h
 	\end{pmatrix}=\left(\begin{array}{c}
 	\varphi' \\
 	-\frac{2}{c}\varphi \varphi'\\
 	0
 	\end{array}\right).
 	$$
 	This has solutions, which are all multiples of
 	\begin{eqnarray*}
 	f &=& \f{1}{c^2} L^{-1} [\vp^3]=\f{1}{2 c(c\be-1)} \vp;  \\
 	g &=& -\f{2}{c^2}\left(\f{\vp^2}{2} + \f{\vp L^{-1}[\vp^3]}{c} \right) = 	- \f{\be}{c(c\be-1)} \vp^2\\
 	h &=&  L_2^{-1} \vp',
 	\end{eqnarray*}
 	where we have used the identity $L\vp=2(\be-\f{1}{c})\vp^3$. This gives a new element
 	$\vec{\xi}\in Ker((\cj\ch)^2)\ominus Ker(\cj\ch)$, namely
 	$$
 	\vec{\xi}:=\left(\begin{array}{c}
 	\f{1}{2 c(c\be-1)} \vp \\ - \f{\be}{c(c\be-1)} \vp^2 \\ L_2^{-1} \vp'
 	\end{array}\right).
 	$$
 		Next, we solve
 		$$
 		\cj\ch \begin{pmatrix}
 		f \\ g\\ h
 		\end{pmatrix}=\begin{pmatrix}
 		0 \\ 0 \\ \varphi
 		\end{pmatrix}.
 		$$
 		We obtain that all  solutions are multiples of the vector
 		\begin{equation}
 		\label{410}
 			f = -L^{-1}\vp, \ \
 			g = \f{2}{c} \vp L^{-1}\vp ,\ \
 			h = 0.
 		\end{equation}
 	We compare this with a similar element,  already present in $Ker(\cj\ch)$.
 	We conclude, that we can consider instead the following new element  $\vec{\eta}\in Ker((\cj\ch)^2)\ominus Ker(\cj\ch)$,
 	$$
 	\vec{\eta}=\left(\begin{array}{c}
 	-L^{-1}\vp  \\ 1+\f{2}{c} \vp L^{-1}\vp \\ 0
 	\end{array}\right)-\left(\begin{array}{c}
 	-L^{-1}\vp  \\ \f{2}{c} \vp L^{-1}\vp \\ 0
 	\end{array}\right)=\left(\begin{array}{c}
 	0 \\ 1 \\ 0
 	\end{array}\right).
 	$$
 	Finally, we solve the equation for the third eigenvector, with unknown $\Psi=\left(\begin{array}{c} \Psi_1 \\ \Psi_2 \\ \Psi_3  \end{array}\right)$
 	\begin{equation}
 	\label{16}
 	\cj\ch \Psi=  \begin{pmatrix}
 	-L^{-1}\vp \\ 1+\f{2}{c} \vp L^{-1} \vp \\ 0
 	\end{pmatrix}.
 	\end{equation}
 	Taking into account that $\cj\ch \Psi=\left(\begin{array}{c} L_2\Psi_3 \\ * \\ *  \end{array}\right)$.
 	This necessitates the  solvability condition $L^{-1}\vp\perp Ker[L_2]=span[\vp]$. This means that as long as $\dpr{L^{-1} \vp}{\vp}\neq 0$, there are no further elements of
 	$Ker((\cj\ch)^2)\ominus Ker(\cj\ch)$.
 	All in all, we have established that
 	\begin{equation}
 	\label{420}
 		Ker((\cj\ch)^2)\ominus Ker(\cj\ch)=span[\vec{\xi}, \vec{\eta}].
 	\end{equation}
 Next, we show that
 \begin{equation}
 \label{430}
 Ker((\cj\ch)^3)\ominus Ker((\cj\ch)^2)=\{0\}.
 \end{equation}
 	Note that combining \eqref{430} and \eqref{420} with \eqref{405}, yields the formula \eqref{402}. So, it remains to show \eqref{430}. To this end, we need to show that the equation
 	\begin{equation}
 	\label{445}
 	\zeta_1 \vec{\xi}+\zeta_2 \vec{\eta}=\cj \ch \Psi= \left(\begin{array}{c}
 	L_2 \Psi_3 \\ 2\p_x(\vp \Psi_1+\f{c}{2} \Psi_2) \\ *
 	\end{array}\right)
 	\end{equation}
 	has no solutions if $(\zeta_1, \zeta_2)\neq (0,0)$. Note that
 	 the first equation in \eqref{445} reads  $L_2 \Psi_3=\f{\zeta_1}{2c(c\be-1)}\vp$. As $Ker(L_2)=span[\vp]$, this forces a solvability condition, $\dpr{\vp}{\f{\zeta_1}{2c(c\be-1)}\vp}=0$, which is impossible, unless $\zeta_1=0$.
 	 Now that we know that $\zeta_1=0$, the second equation in \eqref{445} reads
 	$$
 	2\p_x(\vp \Psi_1+\f{c}{2}\Psi_2)=\zeta_2.
 	$$
 This implies $\vp \Psi_1+\f{c}{2}\Psi_2	= \f{\zeta_2}{2} x+const$. The  left hand side of this identity is   $2T$ periodic, while the right-hand side  is never $2T$ periodic, unless $\zeta_2=0$. Thus, we conclude that $\zeta_2=0$ as well, which establishes \eqref{430}.

 This completes the proof of Proposition \ref{prop1}.
 \end{proof}
  Next, we compute the Morse index of $\ch$.
 \subsection{Morse index of $\ch$}
 In the next Proposition we compute the Morse index of $\ch$.
 \begin{proposition}
 	\label{prop:12}
 	We have the following formula for the Morse index $n(\ch)$,
 	\begin{itemize}
 		\item If  $\vp$ is the dnoidal wave given by \eqref{2.4}, then $n(\ch)=1$.
 		\item For the snoidal case, i.e. $\vp$  is given by \eqref{2.8}, we have  $n(\ch)=3$.
 	\end{itemize}
 	
 \end{proposition}
 \begin{proof}
 	Denote $\ch_0:=\left( \begin{array}{ll} L_1&\varphi \\ \varphi & \frac{c}{2} \end{array} \right)$. Clearly, $n(\ch)=n(\ch_0)+n(L_2)$. Taking into account the computation of $n(L_2)$  in Proposition \ref{prop:spec} (which yields $n(L_2)=0$ in the dnoidal case and $n(L_2)=2$ in the snoidal case), it remains to show that $n(\ch_0)=1$, in both cases under consideration. 	
 	
 	To this end, observe that we have the following expression for the quadratic form associated to $\ch_0$,
 	\begin{equation}\label{s2}
 	\begin{array}{ll}
 	\langle \ch_0 \left(\begin{array}{c} f \\ g \end{array} \right), \left(\begin{array}{c} f \\ g \end{array} \right)\rangle &=\langle L_1f,f\rangle+2\langle \varphi f,g\rangle+\frac{c}{2}\langle g,g\rangle\\
 	&= \langle Lf,f\rangle +\int_{-T}^{T}{\left[ \sqrt{\frac{2}{c}}f+\sqrt{\frac{c}{2}}g\right]^2}dx.
 	\end{array}
 	\end{equation}
 	First, we confirm  that $\ch_0$ has  at least one  negative eigenvalue. Recall from Proposition \ref{prop:spec}, that  $n(L)=1$.  Let us denote by $h$ the eigenfunction of $L$ corresponding to the negative eigenvalue. For $f=h$ and $g:=-\frac{2}{c}h$ in (\ref{s2}), we get
 	$$
 \langle \ch_0 \left(\begin{array}{c} f \\ g \end{array} \right), \left(\begin{array}{c} f \\ g \end{array} \right)\rangle  = \langle Lh,h\rangle <0.
 	$$
 	Hence $\ch_0$ has a negative eigenvalue. Thus, selecting $f\perp h$ and using the max-min characterization of eigenvalues, we have that the second smallest eigenvalue $\la_1$ satisfies the estimate
 	$$
 	\lambda_1(\ch_0) \geq  \inf_{(f,g)\perp (h,0):\|f\|^2+\|g\|^2=1}
 	\langle \ch_0 \left(\begin{array}{c} f \\ g \end{array} \right), \left(\begin{array}{c} f \\ g \end{array} \right)\rangle\geq \inf_{f\perp h, \|f\|\leq 1} \dpr{Lf}{f}\geq 0,
 	$$
 	since $L$ has $n(L)=1$ and so, $\inf_{f\perp h} \dpr{Lf}{f}\geq 0$.
 	That is, $n(\ch_0)=1$.
 \end{proof}
 \section{Stability analysis of the waves}
 \label{sec:3}
 We start by analyzing the stability of the dnoidal waves. Our starting point is the the instability Krein index count \eqref{ih2}. Thus, it remains to determine the Morse index of the matrix $D$ associated with  it.
 Recall that, under the assumption \eqref{ass2}, we have identified
 $$
\vec{\psi}_1=\begin{pmatrix}
\f{1}{2c(c\be-1)} \vp \\
- \f{\be}{c(c\be-1)} \vp^2 \\
\\ L_2^{-1} \vp'.
\end{pmatrix}; \  \vec{\psi}_2= \begin{pmatrix}
-L^{-1}\vp\\
\f{2}{c} \vp L^{-1}\vp\\
0
\end{pmatrix}; \ \vec{\psi}_3=\begin{pmatrix}
0 \\ 1 \\ 0
\end{pmatrix}.
 $$
  so that 	$gKer(\cj\ch)\ominus Ker(\ch)=span[\vec{\psi}_1, \vec{\psi}_2, \vec{\psi}_3]$. By direct computations, we have
 $$\mathcal{H}\vec{\psi}_1=\left( \begin{array}{cc} 0 \\ -\frac{1}{2c} \vp^2 \\ \vp' \end{array} \right); \ \ \mathcal{H}\vec{\psi}_2=\left( \begin{array}{cc} -\vp \\ 0 \\ 0\end{array} \right); \ \  \ch \vec{\psi}_3=\begin{pmatrix}
 \vp \\ \f{c}{2} \\ 0
 \end{pmatrix}
 $$
 and
 \begin{eqnarray}
 \label{103}
 D_{11} &=&\langle\mathcal{H}\psi_1, \psi_1\rangle = \langle L_2^{-1}\vp',\vp'\rangle+\frac{\beta}{2c^2(c\beta-1)}\langle \vp^2,\vp^2\rangle\\
 \label{104}
 D_{12}  = D_{21}  &=&  \langle\mathcal{H}\vec{\psi}_1,\vec{\psi}_2\rangle=-\frac{1}{2c(c\beta-1)}\langle \vp, \vp\rangle\\
 \label{105}
 D_{22} &=&\langle\mathcal{H}\vec{\psi}_2, \vec{\psi}_2\rangle=\langle L^{-1}\vp, \vp\rangle, D_{3 3} = \dpr{\ch \vec{\psi}_3}{\vec{\psi_3}}= c T\\
\label{107}
 D_{1 3}=D_{3 1} &=& \dpr{\ch \vec{\psi}_1}{\vec{\psi}_3} =  -\f{1}{2c} \dpr{\vp}{\vp} \\
 \label{108}
 D_{2 3} = D_{3 2} &=&   \dpr{\ch \vec{\psi}_2}{\vec{\psi}_3}=0.
\end{eqnarray}
 \subsection{Dnoidal waves}

 According to instability index count formula \eqref{ih2} and
 Proposition \ref{prop:12}, which  implies   that $n(\ch)=1$, the stability analysis reduces to  establishing  that $n(D)=1$. Indeed, in such a case, the right-hand side of \eqref{ih2} is zero, thus would rule out all potential instabilities on the left-hand side.

 We proceed to evaluating the elements of the matrix $D$. In fact, we shall need to only compute $D_{22}=\dpr{L^{-1} \vp}{\vp}$, which we will now show is negative.
 To this end, start with the identity $L\varphi'=0$.  In order to construct the Green's function for the operator $L$, we need a solution  $\psi: L \psi=0$. In principle, the following function provides such a solution
  \begin{equation}
 	\label{m:14} 
 	 \psi(x)=\varphi'(x)\int^{x}{\frac{1}{\varphi'^2(s)}}ds, \; \; \left| \begin{array}{cc} \varphi'& \psi \\ \varphi'' & \psi'\end{array}\right|=1
\end{equation}
 Unfortunately, as $\vp'$ has zeros in the interval of integration, this integral is  not well-defined. Instead, we use the standard roundabout way of making the definition of such integral well-defined, which involves integration by parts.   
Specifically,  we proceed by  using the identities
 $$
 \frac{1}{cn^2(y,\kappa)}=\frac{1}{dn(y, \kappa)}\frac{\partial}{\partial_y}\frac{sn(x, \kappa)}{cn(y, \kappa)}, \; \;
 \frac{1}{sn^2(y,\kappa)}=-\frac{1}{dn(y, \kappa)}\frac{\partial}{\partial_y}\frac{cn(x, \kappa)}{sn(y, \kappa)}
 $$
 Integrating by parts yields the alternative, well-defined expression for $\psi$, which is formally equivalent to \eqref{m:14}, 
 \begin{equation}
 	\label{m:15} 
 \psi(x) =\frac{1}{\alpha^2\kappa^2\varphi_0}\left[\frac{1-2sn^2(\alpha x, \kappa )}{dn (\alpha x, \kappa) }
 - \alpha \kappa^2sn(\alpha x, \kappa)cn(\alpha x, \kappa)  \int_{0}^{x}{\frac{1-2sn^2(\alpha s, \kappa )}{dn^2(\alpha s, \kappa)}}ds\right].
\end{equation}
 Thus, we may construct the Green's function as follows
 $$
 L^{-1}f=\varphi'\int_{0}^{x}{\psi(s)f(s)}ds-\psi(s)\int_{0}^{x}{\varphi'(s)f(s)}s+C_f\psi(x),
 $$
 where $C_f$ is chosen, so that $L^{-1}f$ has the same  period   as $\varphi$. After integrating by parts, we get
 \begin{equation}\label{L1}
 \langle L^{-1}\varphi , \varphi\rangle = -\langle \varphi^3, \psi \rangle +\frac{\varphi^2(T)+\varphi(0)^2}{2}\langle \varphi, \psi \rangle +C_{\varphi}\langle \varphi , \psi \rangle .
 \end{equation}
 Integrating by parts yields
 $$\langle \psi'', \varphi\rangle=2\psi'(T)\varphi(T)+\langle \psi, \varphi''\rangle.$$
 Using that $L\varphi=2(\beta-\frac{1}{c})\varphi^3$, we get
 $$
 \langle \psi, \varphi^3\rangle=\frac{c}{c\beta-1}\psi'(T)\varphi(T).
 $$
 Using that $\int_{0}^{K(\kappa)}{\frac{1-2 sn^2(x)}{dn^2(x)}}dx=\frac{1}{\kappa^2(1-\kappa^2)}[2(1-\kappa^2)K(\kappa)-(2-\kappa^2)E(\kappa)]$, we get
 \begin{equation}\label{L2}
 \begin{array}{ll}
 \langle \varphi, \psi \rangle =\frac{1}{\alpha^3 \kappa^2}[E(\kappa)-K(\kappa)] \\
 \\
 \langle \varphi^3, \psi \rangle =\frac{1}{\alpha}\frac{c}{c\beta-1} [2(1-\kappa^2)K(\kappa)-(2-\kappa^2)E(\kappa)] \\
 \\
 C_{\varphi}=-\frac{\varphi''(T)}{2\psi '(T)}\langle \varphi, \psi \rangle+\frac{\varphi^2(T)-\varphi^2(0)}{2}.
 \end{array}
 \end{equation}
 \begin{figure}
 	\centering
 	\includegraphics[width=0.7\linewidth]{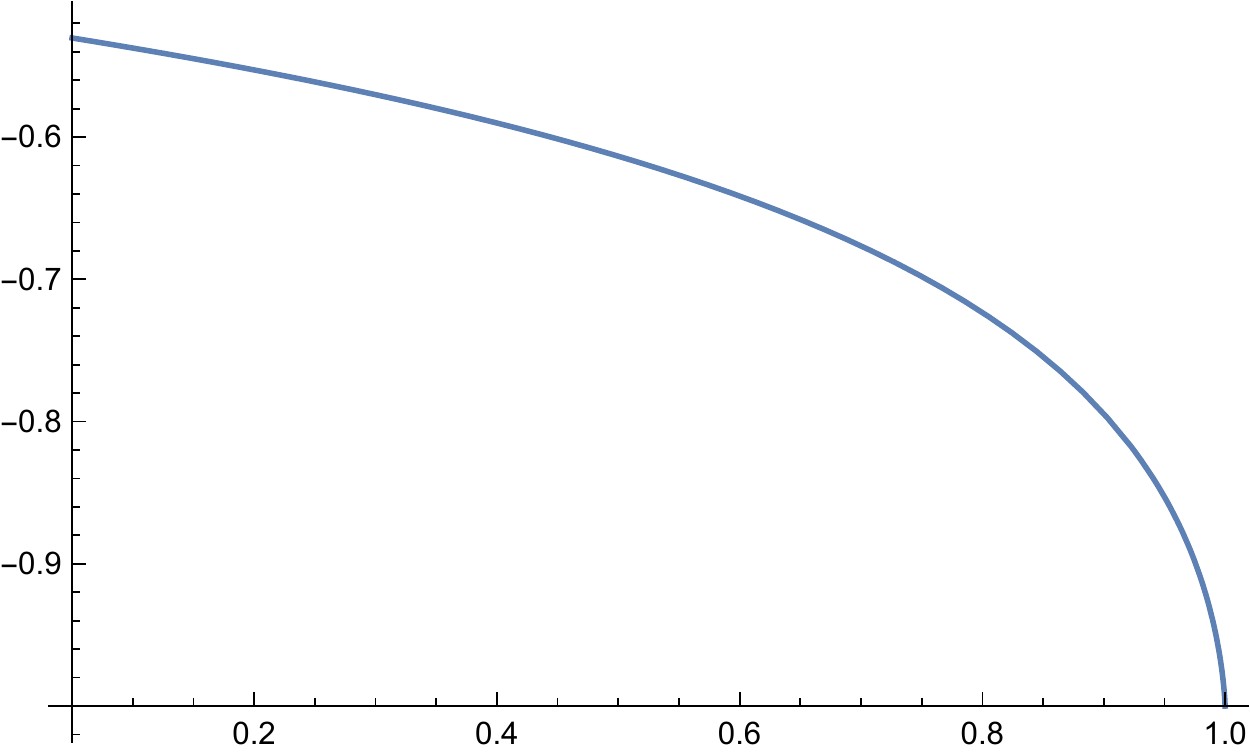}
 	\caption{Graph of 	$\ka\to \frac{E^2(\kappa)-(1-\kappa^2)K^2(\kappa)}{2(1-\kappa^2)K(\kappa)-(2-\kappa^2)E(\kappa)}$}
 	\label{fig:pic1}
 \end{figure}
 Taking into account $\frac{\varphi_0^2}{\alpha^2}=\frac{2c}{1-c\beta}$, we get
 \begin{equation}
 \label{L:12}
 D_{22}=\langle L^{-1} \varphi , \varphi \rangle =\frac{1}{\alpha}\frac{1}{\frac{1}{c}-\beta}\frac{E^2(\kappa)-(1-\kappa^2)K^2(\kappa)}{2(1-\kappa^2)K(\kappa)-(2-\kappa^2)E(\kappa)}<0,
 \end{equation}
 see Figure \ref{fig:pic1}, so in particular, the condition \eqref{ass2} is satisfied. Also, since $D_{22}<0$ for all values of the parameters, it is clear that $D_{2 2}=\dpr{D e_2}{e_2}\leq \inf_{\xi\in \rthree:\|\xi\|=1} \dpr{ D\xi}{\xi}$, whence\footnote{By the way,  by \eqref{ih2} this actually implies that $n(D)=1$. }  $n(D)\geq 1$. As discussed, this implies that  the dnoidal waves are spectrally stable.
 	
 \subsection{Snoidal waves}
 According to the formula \eqref{103}, \eqref{104}, \eqref{105} and \eqref{107}, we shall need to compute
 $\dpr{L_2^{-1} \vp'}{\vp'}$, $\dpr{L^{-1}\vp}{\vp}$ and $\int \vp^2, \int \vp^4$.

 To this end,  we start with the computation of   $\langle L^{-1}\vp,\vp\rangle$.
 We have $L\vp'=0$ and $L\psi=0$, where $\psi(x)=\vp'(x)\int^x{\frac{1}{\vp'^2(s)}}ds$. Using that
 $$
 \frac{1}{cn^2(\alpha x)}=\frac{1}{\alpha dn(\alpha x)}\frac{\partial}{\partial x}\frac{sn(\alpha x)}{cn(\alpha x)},
 $$
 we get the odd function $\psi$
 $$
 \psi (x)=\frac{1}{\vp_0\alpha^2(1-\kappa^2)}\left[ sn(\alpha x)-\alpha \kappa^2cn(\alpha x)dn(\alpha x)\int_{0}^{x}{\frac{1+sn^2(\alpha s)}{dn^2(\alpha s)}}ds\right].
 $$
 Integration by parts yields the formulas
 \begin{eqnarray*}
 & & \langle L^{-1}\vp,\vp\rangle=-\langle \vp^3,\psi\rangle+C_{\vp}\langle \vp, \psi\rangle, \\
 & & \langle \psi'',\vp\rangle=-2\vp'(T)\psi(T)+\langle \psi,\vp''\rangle.
 \end{eqnarray*}
 A direct calculation shows that
 $L\vp=2\left(\beta-\frac{1}{c}\right)\vp^3$, whence
 $$
 \langle \vp^3,\psi\rangle =-\frac{c}{c\beta-1}\vp'(T)\psi(T).
 $$
 Now, we have the relations
 $$
 \left\{ \begin{array}{ll}
 \psi(T)=\frac{\kappa^2}{\vp_0\alpha^2(1-\kappa^2)}\int_{0}^{2K(\kappa)}{\frac{1+sn^2(x)}{dn^2(x)}}dx\\
 \\
 \vp'(T)=-\vp_0\alpha, \ \
 C_{\vp}=-\frac{\vp'(T)}{2\psi(T)}\langle \vp , \psi\rangle. \\
 \\
 \alpha^2=-\frac{\sigma}{1+\kappa^2}, \ \
 \vp_0^2=\frac{2c\sigma \kappa^2}{(1-c\beta)(1+\kappa^2)}
 \end{array} \right.
 $$
 Integration by parts allows us to compute
 $$
 \langle \vp,\psi\rangle=\frac{1}{\alpha^3(1-\kappa^2)}\left[ \int_{0}^{2K(\kappa)}{sn^2(x)}dx+\int_{0}^{2K(\kappa)}{\frac{1+sn^2(x)}{dn^2(x)}}dx-2K(\kappa)\right].
 $$
 Putting all this together, we have
 $$\left\{ \begin{array}{ll}
 \langle \vp^3, \psi\rangle=\frac{1}{\alpha}\frac{c}{c\beta-1}\frac{\kappa^2}{1-\kappa^2}\int_{0}^{2K(\kappa)}{\frac{1+sn^2(x)}{dn^2(x)}}dx\\
 \\C_{\vp}\langle \vp,\psi\rangle=\frac{1}{\alpha}\frac{c}{c\beta-1}\frac{1}{(1-\kappa^2)\int_{0}^{2K(\kappa)}{\frac{1+sn^2(x)}{dn^2(x)}}dx}
 \left[ \int_{0}^{2K(\kappa)}{sn^2(x)}dx+\int_{0}^{2K(\kappa)}{\frac{1+sn^2(x)}{dn^2(x)}}dx-2K(\kappa)\right]^2
 \end{array} \right.
 $$
whence finally
 \begin{equation}
 \label{115}
  \langle L^{-1}\vp,\vp\rangle=\frac{1}{\alpha}\frac{1}{(\beta-\f{1}{c})}F(\ka),
 \end{equation}
 where
 \begin{eqnarray*}
 F(\ka) &=&\left[
 \frac{\left(\int_{0}^{2K(\kappa)}{sn^2(x)}dx+\int_{0}^{2K(\kappa)}{\frac{1+sn^2(x)}{dn^2(x)}}dx-2K(\kappa)\right)^2}
 {(1-k^2)\int_{0}^{2K(\kappa)}{\frac{1+sn^2(x)}{dn^2(x)}}dx}-\f{\kappa^2}{1-\ka^2)} \int_{0}^{2K(\kappa)}{\frac{1+sn^2(x)}{dn^2(x)}}dx\right]\\
 &=& 2 K(\ka)+2 E(\ka) \left(-1+\f{\ka^2 E(\ka)}{(\ka^2+1) E(\ka) - (1-\ka^2) K(\ka)} \right).
 \end{eqnarray*}
  We have plotted it, using \textsc{Mathematica},   see Figure \ref{pic2}. From this,
 it becomes clear that \\ $\dpr{L^{-1} \vp}{\vp}>0$. In particular, the condition \eqref{ass2} holds, whence the conclusions of Proposition \ref{prop1} hold.
 \begin{figure}
 	\centering
 	\includegraphics[width=0.7\linewidth]{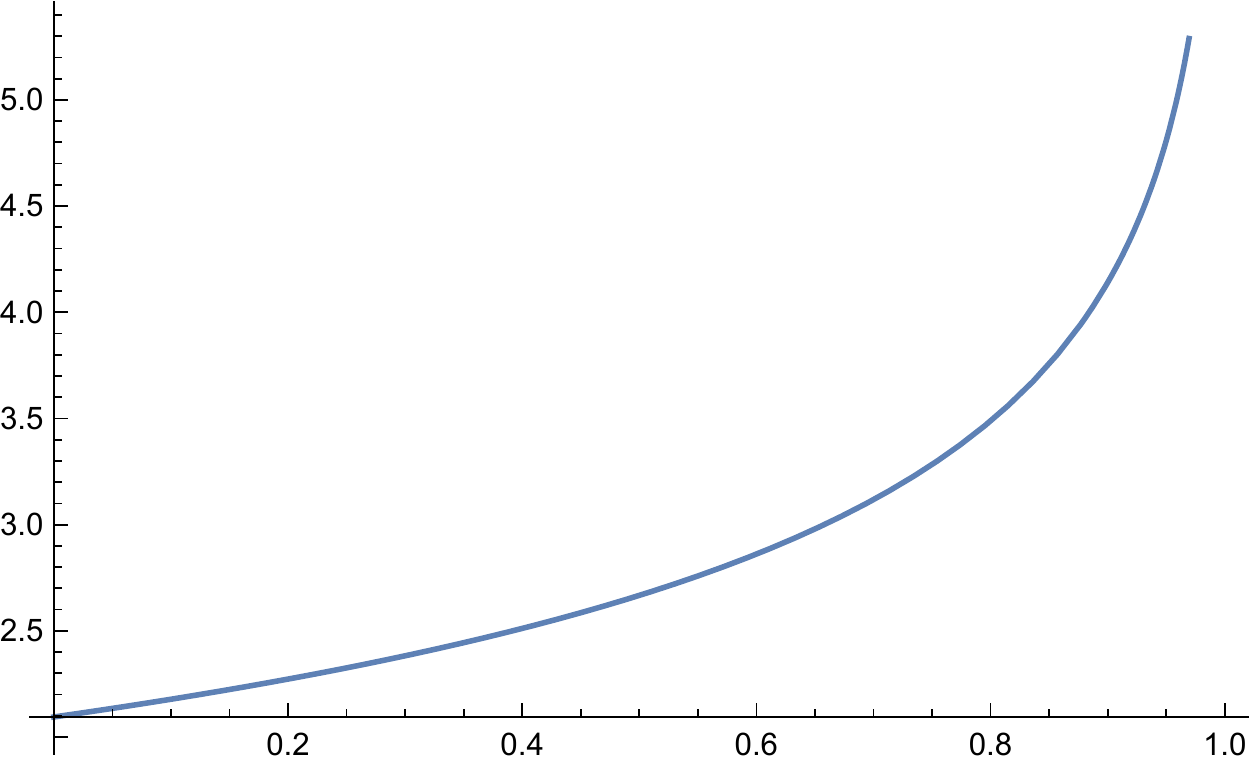}
 	\caption{Graph of 	$F(\ka)$}
 	\label{pic2}
 \end{figure}

 We will now compute
 $\langle L_2^{-1}\vp',\vp'\rangle$.
 We have $L_2\varphi=0$ and $\psi=\varphi \int^{x}{\frac{1}{\varphi^2}}ds$ is also solution of $L_2\psi=0$. Using the identity
 $$
 \frac{1}{sn^2(y,\kappa)}=-\frac{1}{\alpha dn(y, \kappa)}\frac{\partial}{\partial_y}\frac{cn(x, \kappa)}{sn(y, \kappa)}$$
 and integration by parts,  we can alternatively express $\psi$ as  follows
 $$
 \psi(x) =-\frac{1}{\alpha \vp_0}\left[\frac{cn(\alpha x)}{dn(\alpha x)}
 - \alpha \kappa^2sn(\alpha x, \kappa) \int_{0}^{x}{\frac{cn^2(\alpha s, \kappa )}{dn^2(\alpha s, \kappa)}}ds\right].
 $$
 Using that $\varphi$ is odd function and $\psi$ is even function, we get
 \begin{eqnarray*}
 \langle L_2^{-1}\varphi', \varphi'\rangle &=& -\int_{-T}^{T}{\varphi^2\varphi'\psi}dx+C_{\varphi'}\int_{-T}^{T}{\varphi'\psi}dx, \\
C_{\varphi'} &=& -\frac{\varphi'(T)}{2\psi'(T)}\int_{-T}^{T}{\varphi'\psi}dx.
 \end{eqnarray*}
 Hence
 $$
 \langle L_2^{-1}\varphi', \varphi'\rangle=-\int_{-T}^{T}{\varphi^2\varphi'\psi}dx-\frac{\varphi'(T)}{2\psi'(T)}\left( \int_{-T}^{T}{\varphi'\psi}dx \right)^2.
 $$
 In addition, we have
 $$\left\{ \begin{array}{ll}
 \vp'(T)=-\alpha \vp_0, \ \
 \psi'(T)=-\frac{\kappa^2}{\vp_0}\int_{0}^{2K(\kappa)}{\frac{cn^2x}{dn^2x}}dx\\
 \\
 \int_{-T}^{T}{\vp'(x)\psi(x)}dx=-\frac{1}{\alpha}\left[ \int_{0}^{2K(\kappa)}{cn^2(x)}dx+\int_{0}^{2K(\kappa)}{\frac{cn^2(x)}{dn^2(x)}}dx\right]\\
 \\
 \int_{-T}^{T}{\vp^2\vp'(x)\psi(x)}dx=-\frac{\vp_0^2}{\alpha}\left[ 2\int_{0}^{2K(\kappa)}{sn^2(x)cn^2(x)}dx+\frac{\kappa^2}{2}\int_{0}^{2K(\kappa)}{\frac{sn^4(x)cn^2(x)}{dn^2(x)}}dx\right]\\
 \\
 \alpha^2=-\frac{\sigma}{1+\kappa^2}, \ \
 \vp_0^2=\frac{2(-\sigma) \kappa^2}{(\beta-\f{1}{c})(1+\kappa^2)}
 \end{array} \right.
 $$
 Putting all this together, we get
 \begin{eqnarray*}
  \langle L_2^{-1}\varphi', \varphi'\rangle &=& \frac{\vp_0^2}{\alpha}\left[ 2\int_{0}^{2K(\kappa)}{sn^2(x)cn^2(x)}dx+\frac{\kappa^2}{2}\int_{0}^{2K(\kappa)}{\frac{sn^4(x)cn^2(x)}{dn^2(x)}}dx \right]-\\
  &-& \frac{\vp_0^2}{\alpha} \frac{\left( \int_{0}^{2K(\kappa)}{cn^2(x)}dx+\int_{0}^{2K(\kappa)}{\frac{cn^2(x)}{dn^2(x)}}dx\right)^2}
  {2\kappa^2\int_{0}^{2K(\kappa)}{\frac{cn^2x}{dn^2x}}dx}
 \end{eqnarray*}
 Finally,  we have
 \begin{eqnarray*}
 \dpr{\vp}{\vp}  &=& \frac{2\vp_0^2}{\alpha}\int_{0}^{2K(\kappa)}{sn^2(x)}dx,\\
\dpr{\vp^2}{\vp^2} &=& \frac{2\vp_0^4}{\alpha}\int_{0}^{2K(\kappa)}{sn^4(x)}dx.
 \end{eqnarray*}
We now compute $\det(D)$, in the regime $\be=\f{1}{c}+\eps, 0<\eps<<1$.  We will establish the following proposition, regarding the matrix $D$,  introduced in \eqref{103}, \eqref{104}, \eqref{105}, \eqref{107}, \eqref{108}.
\begin{proposition}
	\label{prop:18}
	Fix $c\neq 0,  \si<0$. Then, there exists $\eps_0=\eps_0(c,\si)>0$, so that for all $0<\eps<\eps_0$ and $\be=\f{1}{c}+\eps$, we have that $det(D)>0$.
\end{proposition}
Before we proceed with the proof of Proposition \ref{prop:18}, let us finish the proof of Theorem \ref{theo:20}. That is, we show that the snoidal waves are spectrally unstable.

We argue  as follows - for very small $\eps$, we have from Proposition \ref{prop:18} that $\det(D)>0$, whence the symmetric matrix $D$ has  either two negative eigenvalues and a positive one (in which case $n(D)=2$), or three positive eigenvalues or $n(D)=0$.

By \eqref{ih2}, we conclude that either $k_{Ham}=3-2=1$ or $k_{Ham.}=n(\cl)-n(D)=3-0=3$.
This implies that there is at least one real instability. In fact, for systems with $k_{Ham}=1$, this is obvious. If $k_{Ham.}=3$, the possibilities are as follows - three real instabilities, one real instability and two complex/oscillatory instabilities and one real instability and a pair of purely imaginary eigenvalues of negative Krein signature. Unfortunately, the instability index theory outlined in Section \ref{sec:2.1} does  not allow us to specify precisely which situation we finds ourselves in, even for $\eps<<1$. We claim that we can nevertheless  confirm that the waves are unstable, in the sense that the eigenvalue problem \eqref{124} has at least one positive eigenvalue.

To this end, consider the parameters $c, \si$ fixed, and $\be$ as a bifurcation parameter.  We start with the observation made above, that  for small $0<\eps<<1$ (that is $\be$ slightly bigger than $\f{1}{c}$), we have at least one real unstable eigenvalue. Allowing the parameter $\be>\f{1}{c}$ to increase, the Krein index may of course change, since our analysis showed that $n(D)=0$ or $n(D)=2$  only for $0<\eps<<1$. But regardless of that, there will always be at least one real instability. This is due to Proposition \ref{prop1} which asserts that the eigenvalue at zero is of algebraic multiplicity five for all values of the parameters, with three eigenvectors and two generalized eigenvectors described there. The only scenario for the real instability present at $\eps<<1$ to become stable is by passing through the zero generalized  eigenspace for some intermediate value of $\be$, which would have been detected by our analysis in Proposition \ref{prop1}. As we have shown, this does not happen. Thus, the real and positive eigenvalue is present for all $\be>\f{1}{c}$, and the snoidal waves $\vp$ are unstable. This completes the proof of Theorem \ref{theo:20} and it remains to establish Proposition \ref{prop:18}.
 \subsection{Proof of Proposition \ref{prop:18}}
 We first calculate $\det(D)$. By the specifics of it, see \eqref{103}, \eqref{104}, \eqref{105}, \eqref{107}, \eqref{108}, we have
 $$
 \det(D)=D_{33} \det(\tilde{D}) - \left(\f{1}{2c}\int\vp^2\right)^2 D_{22},
 $$
 where $\tilde{D}=\left(\begin{array}{cc} D_{11} & D_{12} \\ D_{12} & D_{22}
 \end{array}\right)$.
Taking into account the form of $\vp_0^2=const. \eps^{-1}$, we have
$$
\dpr{L_2^{-1} \vp'}{\vp'}=const. \eps^{-1} + O(\eps^{-2}),
$$
 while $\dpr{\vp^2}{\vp^2}=const. \eps^{-2}+O(\eps^{-1})$. So, we can conclude that
 $$
 D_{11}=\f{\be}{2c^2 (c\be-1)} \dpr{\vp^2}{\vp^2}+\dpr{L_2^{-1} \vp'}{\vp'}=\f{\be}{2c^2 (c\be-1)} \dpr{\vp^2}{\vp^2}+O(\eps^{-1}),
 $$
  whence
\begin{eqnarray*}
& & \det(\tilde{D}) = D_{11} D_{22}- D_{12}^2 =\\
&=&  \f{\be \vp_0^4}{\al c (c\be-1)} \left(\int_0^{2K(k\ka)} sn^4(x) dx\right) \f{F(\ka)}{\al (c\be-1)} -   \f{\vp_0^4}{\al^2 c^2(c\be-1)^2}  \left(\int_0^{2K(k\ka)} sn^2(x) dx\right)^2+O(\eps^{-2})\\
&=& \f{\vp_0^4}{\al^2 (c\be-1)^2} \left[\f{\be}{c} \left(\int_0^{2K(k\ka)} sn^4(x) dx\right) F(\ka) - \f{1}{c^2} \left(\int_0^{2K(k\ka)} sn^2(x) dx\right)^2\right]+O(\eps^{-2}).
\end{eqnarray*}
 Clearly, the first expression is of the form $const. \eps^{-4}$ and hence,  it is dominant in the regime $\be=\f{1}{c}+\eps, 0<\eps<<1$. Furthermore, the assignment $\be=\f{1}{c}+\eps$ allows us to further extract a leading order term as follows
 $$
  \det(\tilde{D}) = \f{\vp_0^4}{\al^2 c^2(c\be-1)^2 } \left[  \left(\int_0^{2K(k\ka)} sn^4(x) dx\right) F(\ka) -  \left(\int_0^{2K(k\ka)} sn^2(x) dx\right)^2\right]+O(\eps^{-3}).
 $$
 We have computed this last function of $\ka$  in \textsc{Mathematica}, and we have obtained the following explicit expression for it
 \begin{eqnarray*}
& & H(\ka) =  \left(\int_0^{2K(k\ka)} sn^4(x) dx\right) F(\ka) -  \left(\int_0^{2K(k\ka)} sn^2(x) dx\right)^2 = \\
 &=& \frac{(2 (\ka^2+2) K(k)-4 (\ka^2+1) E(k)) \left(2 K(k)+2 E(k) \left(\frac{\ka^2 E(k)}{(\ka^2-1) K(k)+(\ka^2+1) E(k)}-1\right)\right)-
 	12 ( E(k)- K(k))^2}{3 \ka^4}
 \end{eqnarray*}
 Plotting this leads to the conclusion $H[\ka]>0$ that, see Figure \ref{pic3}.
 \begin{figure}
 	\centering
 	\includegraphics[width=0.7\linewidth]{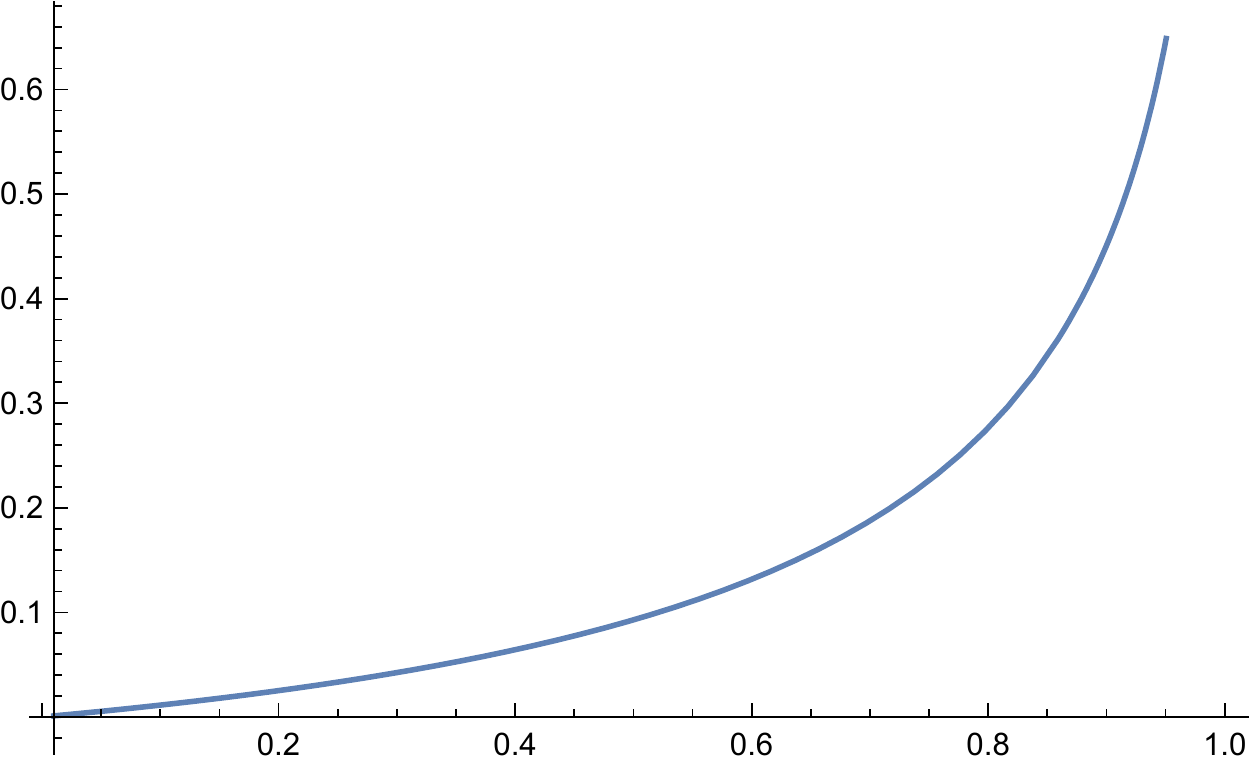}
 	\caption{Graph of 	$H(\ka)$}
 	\label{pic3}
 \end{figure}
  Thus, to a leading order
 $$
 \det(\tilde{D})=C(k,\si,c) \eps^{-4}+O(\eps^{-3}),
 $$
as $\eps: 0<\eps<<1$. In addition, observe that by \eqref{115}, we have that
$$
\left(\f{1}{2c}\int\vp^2\right)^2 D_{22}=O(\eps^{-3}).
$$
Accordingly, we have that
$$
\det(D)=D_{33} \det(\tilde{D}) - \left(\f{1}{2c}\int\vp^2\right)^2 D_{22}=C(k,\si,c) \eps^{-4}+O(\eps^{-3}),
$$
whence $\det(D)>0$, for all small enough $\eps>0$. This completes the proof of Proposition \ref{prop:18}. \\
{\bf Conflict of interest and data availability statement:} On behalf of all authors, the corresponding author states that there is no conflict of interest. We declare that our manuscript has no associated data.

\end{document}